\documentclass[reqno]{amsart}

\usepackage{etex}
\usepackage{hyperref}
\usepackage{amssymb}
\usepackage[all]{xy}
\usepackage{ifthen}
\usepackage{xargs}

\hypersetup{colorlinks=true,linkcolor=blue}

\author{Valery Isaev}

\title{On fibrant objects in model categories}

\newcommand{\newref}[4][]{
\ifthenelse{\equal{#1}{}}{\newtheorem{h#2}[hthm]{#4}}{\newtheorem{h#2}{#4}[#1]}
\expandafter\newcommand\csname r#2\endcsname[1]{#3~\ref{#2:##1}}
\expandafter\newcommand\csname R#2\endcsname[1]{#4~\ref{#2:##1}}
\expandafter\newcommand\csname n#2\endcsname[1]{\ref{#2:##1}}
\newenvironmentx{#2}[2][1=,2=]{
\ifthenelse{\equal{##2}{}}{\begin{h#2}}{\begin{h#2}[##2]}
\ifthenelse{\equal{##1}{}}{}{\label{#2:##1}}
}{\end{h#2}}
}

\newref[section]{thm}{theorem}{Theorem}
\newref{lem}{lemma}{Lemma}
\newref{prop}{proposition}{Proposition}
\newref{cor}{corollary}{Corollary}
\newref{rem}{remark}{Remark}
\newref{exmp}{example}{Example}

\theoremstyle{definition}
\newref{defn}{definition}{Definition}

\newcommand{\we}{\mathcal{W}}
\newcommand{\fib}{\mathcal{F}}
\newcommand{\cof}{\mathcal{C}}
\newcommand{\cat}[1]{\mathbf{#1}}
\newcommand{\C}{\cat{C}}

\newcommand{\I}{\mathrm{I}}
\newcommand{\J}{\mathrm{J}}
\newcommand{\class}[2]{#1\text{-}\mathrm{#2}}
\newcommand{\Iinj}[1][\I]{\class{#1}{inj}}
\newcommand{\Icell}[1][\I]{\class{#1}{cell}}
\newcommand{\Icof}[1][\I]{\class{#1}{cof}}
\newcommand{\Jinj}[1][]{\Iinj[\J#1]}
\newcommand{\Jcell}[1][]{\Icell[\J#1]}
\newcommand{\Jcof}[1][]{\Icof[\J#1]}
\newcommand{\cyli}{i}

\newcommand{\po}[1][dr]{\save*!/#1+1.2pc/#1:(1,-1)@^{|-}\restore}

\begin{document}

\begin{abstract}
In this paper, we study properties of maps between fibrant objects in model categories.
We give a characterization of weak equivalences between fibrant object.
If every object of a model category is fibrant, then we give a simple description of a set of generating cofibrations.
We show that to construct such model structure it is enough to check some relatively simple conditions.
\end{abstract}

\maketitle

\section{Introduction}

Model categories, introduced in \cite{quillen}, are an important tool in homotopy theory.
In this paper, we study properties of maps between fibrant objects.
We will show that weak equivalence between fibrant objects have a simple characterization in terms of generating cofibrations (\rprop{min-we}).
In particular, if every object of a model category is fibrant, then we get a complete characterization of weak equivalences.
Using this characterization, we will show that to construct a model structure in which all objects are fibrant,
it is enough to verify some (relatively) simple conditions (\rthm{main}, \rprop{main-path}, \rprop{main-cyl}).
Examples of applications of these constructions are model structures on strict $\omega$-categories (which was constructed in \cite{folk}) and topological spaces.
Examples of new model structures constructed by means of results of this paper will be given in \cite{alg-models}.

If we fix the class of cofibrations, then the class of weak equivalences which defines such model structure is unique if it exists.
These class of weak equivalences is the minimum possible among model structures with this class of cofibrations.
Thus a model category in which all objects are fibrant is left-determinant (as defined in \cite{left-det}).
But the converse is not true.
For example, the category of simplicial sets (as well as every Grothendieck topos \cite{cisinski})
carries a left-determinant model structure with monomorphisms as cofibrations,
but a model structure on this category in which all objects are fibrant and with this class of cofibrations does not exist.
To see this, consider a cylinder $C$ for the terminal object.
Then $C \amalg_{\Delta[0]} C$ does not have a map from $\Delta[1]$ which
maps faces of this simplex to points of $C$ which were not amalgamated.
Thus $C \amalg_{\Delta[0]} C$ cannot be fibrant.

Thus the class of model structures in which all objects are fibrant is much narrower than the class of left-determinant model structures.
On the other hand model categories in which all objects are fibrant have properties not shared by left-determinant ones.
Often a set of generating trivial cofibrations is defined using cardinality bounds on domains and codomains of maps,
which does not give us a useful description of generating trivial cofibrations.
But if all objects are fibrant, then we can give a simple explicit construction of a set of generating trivial cofibrations (see \rcor{min-cof-gen}).
Also, as we already noted, model categories in which all objects are fibrant have a simple characterization of weak equivalences in terms of generating cofibrations.

In section 2, we recall definitions and basic properties of weak factorization systems and model categories.
In section 3, we prove several results that are related to fibrant objects, fibrations and trivial cofibrations between them.
In section 4, we demonstrate a method of constructing model structures in which all objects are fibrant.

\section{Preliminaries}

In this section, we will introduce a few definitions that we will need later.
We will also recall the definition and basic properties of model categories and prove a slightly generalized versions of some standard lemmas.

\subsection{Homotopy relation}

Let $\C$ be a category and let $V$ be an object of $\C$.
Then we a \emph{cylinder object} for $V$ is an object $C(V)$ together with maps $\cyli_0,\cyli_1 : V \to C(V)$.
If $i : U \to V$ is a morphism of $\C$, then a \emph{relative cylinder object} for $\C$
is a cylinder object $(C_U(V),\cyli_0,\cyli_1)$ for $V$ such that $\cyli_0 \circ i = \cyli_1 \circ i$.
If $\C$ has the initial object $0$, then a relative cylinder object for $0 \to V$ is just a cylinder object for $V$.
A morphism of cylinder objects $C(V_1)$ and $C(V_2)$ is a pair of maps $f : V_1 \to V_2$ and $C(f) : C(V_1) \to C(V_2)$ which commute with $\cyli_0$ and $\cyli_1$.

A \emph{(left) homotopy} (with respect to $C(V)$) between maps $f,g : V \to X$ is a map $h : C(V) \to X$ such that $h \circ \cyli_0 = f$ and $h \circ \cyli_1 = g$.
Maps are \emph{homotopic} if there is a homotopy between them.
Maps are \emph{homotopic relative to $i : U \to V$} (with respect to $C_U(V)$) if there is a homotopy with respect to $C_U(V)$ between them.
Note that maps are homotopic relative to $i$ only if $i \circ f = i \circ g$.
If maps $f$ and $g$ are homotopic, then we write $f \sim g$, and if they are homotopic relative to $i$, then we write $f \sim_i g$.

Let $V,Y$ be objects of a category $\C$ and $R$ some relation on the set $\C(V,Y)$.
Given two morphisms $f : U \to V$ and $g : X \to Y$, we say that $f$ \emph{has the left lifting property (LLP) up to $R$} with respect to $g$,
and $g$ \emph{has the right lifting property (RLP) up to $R$} with respect to $f$ if for every commutative square of the form
\[ \xymatrix{ U \ar[r]^u \ar@{}[dr]|(.7){R} \ar[d]_f & X \ar[d]^g \\
              V \ar[r]_v \ar@{-->}[ur]^h             & Y,
            } \]
there is a dotted arrow $h : V \to X$ such that $h \circ f = u$ and $(g \circ h) R v$.
We say a map $f$ has the right lifting property up to $R$ with respect to
an object $V$ if it has this property with respect to the map $0 \to V$.
Given a morphism $f$ and a set of morphisms $\I$, $f$ has the left (right) lifting property up to $R$
with respect to $\I$ if it has this property with respect to all morphisms in $\I$.
A map has the right (left) lifting property if it has this property up to the equality relation.

Let $R$ be the maximal relation on the set $\C(V,Y)$, that is for every $f_1,f_2 : V \to Y$, we have $f_1\,R\,f_2$.
We will say that $g : X \to Y$ is \emph{pure} with respect to $f : U \to V$ if $g$ has RLP up to $R$ with respect to $f$.
The notion of pure morphism is (formally) similar to the concept of
$\lambda$-pure morphism, used in the theory of accessible categories.

We list a few elementary properties of pure morphisms in the following proposition:

\begin{prop} The following holds in every category $\C$:
\begin{enumerate}
\item If $g$ has RLP up to some relation with respect to $f$, then $g$ is pure with respect to $f$.
\item Pure maps are closed under composition.
\item If $f : X \to Y$ and $g : Y \to Z$ are maps such that $g \circ f$ is pure, then $f$ is also pure.
\item Every split monomorphism is pure with respect to all maps.
\item If a map is pure with respect to itself, then it is a split monomorphism.
\end{enumerate}
\end{prop}

Let $\C$ be a category and let $X$ be an object of $\C$.
A \emph{path object} for $X$ is an object $P(X)$ together with maps $p_0,p_1 : P(X) \to X$.
A \emph{(right) homotopy} between morphisms $f,g : V \to X$ is a map
$h : V \to P(X)$ such that $p_0 \circ h = f$ and $p_1 \circ h = g$.
We say that $f$ and $g$ are right homotopic and write $f \sim^r g$
if there exists a right homotopy $h : V \to P(X)$ between them.
A morphism of path objects $P(X)$ and $P(Y)$ is a pair of maps $f : X \to Y$ and $P(f) : P(X) \to P(Y)$ which commute with $p_0$ and $p_1$.

Note that $\sim$ is reflexive if and only if there exists a map $s : C_U(V) \to V$ such that $s \circ \cyli_0 = s \circ \cyli_1 = id_V$.
In this case we will say that $C_U(V)$ is \emph{reflexive}.
Relation $\sim^r$ is reflexive if and only if there exists a map $t : X \to P(X)$ such that $p_0 \circ t = p_1 \circ t = id_X$.
In this case we will say that $P(X)$ is \emph{reflexive}.
If $P(X)$ is reflexive, then we say that a right homotopy $h : V \to P(X)$ between $f,g : V \to X$
is \emph{constant on $i : U \to V$} if $h \circ i = t \circ f \circ i$.
In this case, we write $f \sim^r_i g$ and say that $f$ and $g$ are \emph{homotopic relative to $i$}.

\subsection{Model categories}

Model categories were introduced in \cite{quillen}.
For an introduction to the theory of model categories see \cite{hirschhorn,hovey}.

\begin{defn} A weak factorization system on a category $\C$ is a pair $(\mathcal{L},\mathcal{R})$
of full subcategories of the category $\C^\to$ of arrows of $\C$ such that
\begin{itemize}
\item Every morphism factors into a map in $\mathcal{L}$ followed by a map in $\mathcal{R}$.
\item $\mathcal{L}$ is the class of maps that have LLP with respect to $\mathcal{R}$.
\item $\mathcal{R}$ is the class of maps that have RLP with respect to $\mathcal{L}$.
\end{itemize}
\end{defn}

\begin{defn}
A \emph{model structure} on a category $\C$ is a choice of three classes of morphisms $\fib$, $\cof$, and $\we$,
called \emph{fibrations}, \emph{cofibrations}, and \emph{weak equivalences} respectively, satisfying the following axioms:
\begin{itemize}
\item $\we$ has 2-out-of-3 property, that is given composable morphisms $f,g$,
    if two out of three morphisms $f$, $g$, $g \circ f$ are in $\we$, so is the third.
\item $(\cof, \fib \cap \we)$ and $(\cof \cap \we, \fib)$ are weak factorization systems.
\end{itemize}
A \emph{model category} is a complete and cocomplete category equipped with a model structure.
\end{defn}

A map is called a \emph{trivial fibration} if it belongs to $\fib \cap \we$,
and it is called a \emph{trivial cofibration} if it belongs to $\cof \cap \we$.
An object $X$ is called \emph{fibrant} if the morphism $X \to 1$ to the terminal object is a fibration,
and it is called \emph{cofibrant} if the morphism $0 \to X$ from the initial object is a cofibration.

Let $\C$ be a model category.
Then we can define a relative cylinder object $C_U(V)$ for every $i : U \to V$.
Let $[\cyli_0,\cyli_1] : V \amalg_U V \to C(V)$, $s : C(V) \to V$ be a factorization of $[id_V,id_V] : V \amalg_U V \to V$ into a cofibration and a trivial fibration.
The homotopy relation corresponding to this cylinder object is always reflexive and symmetric.

Every object $X$ in a model category $\C$ has a reflexive path object which is defined as a factorization of the diagonal $\langle id_X, id_X \rangle : X \to X \times X$.
The following proposition is standard:
\begin{prop}[path-cyl]
If $i : U \to V$ is a cofibration and $X$ is a fibrant object, then maps $f,g : V \to X$ are left homotopic relative to $i$ if and only if they are right homotopic relative to $i$.
This homotopy relation is an equivalence relation.
\end{prop}

A map $f : X \to Y$ is \emph{an inclusion of a deformation retract} if there is
a map $g : Y \to X$ such that $g \circ f = id_X$ and $f \circ g \sim id_Y$.
A map $f : X \to Y$ is \emph{an inclusion of a strong deformation retract} if the homotopy is relative to $f$.

The following lemmas generalize standard properties of model categories.

\begin{lem}[hom-ext][Homotopy extension property]
Let $\C$ be a category, and let $i : U \to V$ be a morphism of $\C$.
Suppose that for an object $X$ of $\C$, there exists a path object $p_0,p_1 : P(X) \to X$ such that $p_0$ has RLP with respect to $i$.

Let $i : U \to V \in \cof$, $u : U \to X$, and $v : V \to X$ be maps, and let $h : U \to P(X)$ be a homotopy between $v \circ i$ and $u$.
Then there exists a map $v' : V \to X$ and a homotopy $h' : V \to P(X)$ between $v$ and $v'$ such that $h = h' \circ i$.
\end{lem}
\begin{proof}
Let $h : U \to P(X)$ be a homotopy between $v \circ i$ and $u$.
Consider the following diagram:
\[ \xymatrix{ U \ar[r]^-h \ar[d]_i & P(X) \ar[d]^{p_0} \\
              V \ar[r]_v & X
            } \]
By assumption, we have a lift $h' : V \to P(X)$ which is a required homotopy.
\end{proof}

This lemma can be illustrated as follows:
\[ \xymatrix{ U \ar[r]^u \ar[d]_i \ar@{}[dr]|(.3){\sim^r} & X \\
              V \ar[ur]_v &
            }
\qquad \qquad
   \xymatrix{ U \ar[r]^u \ar[d]_i \ar@{}[dr]|(.62){\sim^r} & X \\
              V \ar@{-->}[ur]^{v'} \ar@/_1pc/[ur]_v &
            } \]
If we have a diagram on the left, then we can find a map $v'$ such that the diagram on the right commutes.
Moreover, if we restrict the homotopy between $v$ and $v'$ on $U$, then we get the original homotopy between $v \circ i$  and $u$.
Let $\sim^{r*}$ be the reflexive transitive closure of $\sim^r$.
Then the previous lemma also holds for $\sim^{r*}$ in place of $\sim^r$:

\begin{lem}[hom-ext-rtc]
Let $\C$ be a category that satisfies conditions of the previous lemma.
Let $i : U \to V \in \cof$, $u : U \to X$, and $v : V \to X$ be maps, and let $h_1, \ldots h_n : U \to P(X)$ be a sequence of homotopies
such that $p_1 \circ h_j = p_0 \circ h_{j+1}$ for every $1 \leq j < n$, $p_0 \circ h_1 = v \circ i$ and $p_1 \circ h_n = u$.
Then there exists a map $v' : V \to X$ and a sequence of homotopies $h'_1, \ldots h'_n : V \to P(X)$
such that $p_1 \circ h'_j = p_0 \circ h'_{j+1}$ for every $1 \leq j < n$, $p_0 \circ h'_1 = v$ and $p_1 \circ h'_n = v'$ and $h_j = h'_j \circ i$ for every $1 \leq j \leq n$.
\end{lem}
\begin{proof}
Apply the previous lemma $n$ times.
\end{proof}

Let $\C$ be a category and let $\I$ be a class of morphisms of $\C$.
Then we define $\Iinj$ to be the class of morphisms of $\C$ that has RLP with respect to $\I$,
$\Icof$ to be the class of morphisms of $\C$ that has LLP with respect to $\Iinj$, and
$\Icell$ to be the class of transfinite compositions of pushouts of elements of $\I$.
Elements of $\Icell$ are called \emph{relative $\I$-cell complexes}.
Every relative $\I$-cell complex belongs to $\Icof$.

We say that a set $\I$ of maps of a cocomplete category $\C$ \emph{permits the small object argument}
if the domains of maps in $\I$ are small relative to $\Icell$.

\begin{prop}[][The small object argument]
Let $\C$ be a cocomplete category and $\I$ a set of maps of $\C$ that permits the small object argument.
Then $(\Icell,\Iinj)$ is a weak factorization system.
\end{prop}

\begin{defn}
Let $\C$ be a category with a model structure. Then this model category is
\emph{cofibrantly generated} if there are sets $\I$ and $\J$ of maps of $\C$ such
that they permit the small object argument, $\fib = \Jinj$, and $\fib \cap \we = \Iinj$.

Elements of $\I$ are called \emph{generating cofibrations},
and elements of $\J$ are called \emph{generating trivial cofibrations}.
\end{defn}

\begin{prop}[model-cat]
Suppose that $\C$ is a complete and cocomplete category, $\we$ is a class of morphisms of $\C$, and $\I$, $\J$ are sets of morphisms of $\C$.
Then $\C$ is a cofibrantly generated model category with $\I$ as the set of generating cofibrations,
$\J$ as the set of generating trivial cofibrations, and $\we$ as the class of weak equivalences if and only if the following conditions are satisfied:
\begin{description}
\item[(A1)] $\I$ and $\J$ permit the small object argument.
\item[(A2)] $\we$ has 2-out-of-3 property and is closed under retracts.
\item[(A3)] $\Iinj \subseteq \we$.
\item[(A4)] $\Jcell \subseteq \we \cap \Icof$.
\item[(A5)] Either $\Jinj \cap \we \subseteq \Iinj$ or $\Icof \cap \we \subseteq \Jcof$.
\end{description}
\end{prop}

\section{Properties of model categories}

In this section, we will prove various properties of model categories that are related to fibrant objects.
In particular, we will give a characterization of weak equivalences between fibrant objects.

\subsection{Generating trivial cofibrations}

First, let us show how to construct a set of generating trivial cofibrations from a set of generating cofibrations satisfying some mild hypotheses,
assuming that every object in a model category is fibrant.
Let $\C$ be a category and $\I$ a class of maps of $\C$ together with a relative cylinder object $C_U(V)$ for every $i : U \to V$ in $\I$.
Let us denote by $\J_\I$ the class of maps $\cyli_0 : V \to C_U(V)$ for each $U \to V \in \I$.

\begin{prop}[triv-fib-iinj]
Let $\C$ be a category, and let $\I$ be a class of maps of $\C$.
If $f : X \to Y \in \Jinj[_\I]$ has RLP up to $\sim_i$ with respect to every $i \in \I$, then $f \in \Iinj$.
\end{prop}
\begin{proof}
Suppose we have a commutative square as below.
We need to find a lift $V \to X$ such that both triangles commute.
\[ \xymatrix{ U \ar[d]_{i \in \I} \ar[r] \ar@{}[dr]|(.7){\sim} & X \ar[d]^f \\
              V \ar[r]_v \ar@{-->}[ur]^g                       & Y
            } \]
By assumption, we have a lift $g : V \to X$ together with
a homotopy $h : C_U(V) \to Y$ between $f \circ g$ and $v$.
Since $f$ has RLP with respect to $\cyli_0$, we have a lift $h' : C_U(V) \to X$ in the following diagram:
\[ \xymatrix{ V \ar[r]^g \ar[d]_{\cyli_0}         & X \ar[d]^f \\
              C_U(V) \ar[r]^-h \ar@{-->}[ur]^{h'} & Y.
            } \]
Then $h' \circ \cyli_1$ is a required lift in the original square.
\end{proof}

\begin{cor}[min-cof-gen]
Let $\C$ be an model category in which every object is fibrant.
Suppose that the class of cofibrations is generated by a set $\I$ such that
the domains and the codomains of maps in $\I$ are small relative to $\Icell$.
Then the model structure is cofibrantly generated with $\J_\I$ as a set of generating trivial cofibrations.
\end{cor}
\begin{proof}
This follows from the previous proposition and \rprop{model-cat}.
\end{proof}

\subsection{Weak equivalences between fibrant objects}

The following propositions give useful characterizations of trivial
cofibrations and weak equivalences between fibrant objects.

\begin{prop}[min-triv-cof]
Let $\C$ be a model category.
Any inclusion of a strong deformation retract $f$ is a weak equivalence.
If $f$ is a cofibration between fibrant objects, then the converse is true.
\end{prop}

This characterization of trivial cofibrations is probably well-known.
Similar propositions are proved in \cite{hirschhorn}, but we could not find a reference for this property, so we include a proof for the sake of convenience.

\begin{proof}
First, let us show that any inclusion of a strong deformation retract $f : X \to Y$ is a weak equivalence.
Factor $f$ into a trivial cofibration $i : X \to Z$ followed by a fibration $p : Z \to Y$.
Let $g : Y \to X$ be a retraction of $f$ and let $h : C_X(Y) \to Y$ be a homotopy between $f \circ g$ and $id_Y$.
Consider the following diagram:
\[ \xymatrix{ Y \ar[r]^{i \circ g} \ar[d]_{\cyli_0} & Z \ar[d]^p \\
              C_X(Y) \ar[r]_-h \ar@{-->}[ur]^{h'}   & Y.
            } \]
Since $\cyli_0 : Y \to C_X(Y)$ is a trivial cofibration, we have a lift $h' : C_X(Y) \to Z$.
Then $h' \circ \cyli_1 \circ f = i$ and $p \circ h' \circ \cyli_1 = id_Y$, thus $f$ is a retract of $i$, hence a weak equivalence.

To prove the converse, let us show that every weak equivalence with fibrant domain is pure with respect to cofibrations.
This follows from the fact that every weak equivalence can be factored into a trivial cofibrant and a trivial fibration.
Every trivial fibration is pure with respect to cofibrations, and every trivial cofibration with fibrant domain is a split monomorphism,
hence it is pure with respect to any map.

Now, let $f : X \to Y$ be a trivial cofibration between fibrant objects.
Since $X$ is fibrant, $f$ has a retraction $g : Y \to X$.
By 2-out-of-3 property, $g$ is a weak equivalence.
Since $Y$ is fibrant, $g$ is pure with respect to cofibrations.
Consider the following diagram:
\[ \xymatrix{ Y \amalg_X Y \ar[rr]^-{[f \circ g, id_Y]} \ar[d]_{[\cyli_0,\cyli_1]} & & Y \ar[d]^g \\
              C_X(Y) \ar[rr]_-{g \circ s} & & X
            } \]
Since $g$ is pure, we have a lift, which gives us a homotopy between $f \circ g$ and $id_Y$.
\end{proof}

Now, we need to prove a technical lemma, which we formulate in a general form since we will need it later.

\begin{lem}[we-bot]
Let $\C$ be a finitely cocomplete category, and let $i : U \to V$ and $g : Y \to Z$ be maps of $\C$.
Let $C_U(V)$ be a relative cylinder object for $i$.
Suppose that for every $A \in \{ Y, Z \}$, there exists a path object $p_0,p_1 : P(A) \to A$, which satisfy the following conditions:
\begin{enumerate}
\item For every $A \in \{ Y, Z \}$, $p_0 : P(A) \to A$ has RLP with respect to $i$ and it is pure with respect to $[\cyli_0,\cyli_1] : V \amalg_U V \to C_U(V)$.
\item There exists a morphism of path objects $(g,P(g)) : P(Y) \to P(Z)$.
\item Either $p_1 : P(Z) \to Z$ has RLP with respect to $i$ or there exists a map $s : P(Z) \to P(Z)$ such that $p_0 \circ s = p_1$ and $p_1 \circ s = p_0$.
\end{enumerate}

Let $f : X \to Y$ be a map of $\C$ such that $f$ has RLP up to $\sim^{r*}$ with respect to $U$, and $g \circ f$ has RLP with respect to $i$ up to $\sim_i$.
Then $g$ also has RLP with respect to $i$ up to $\sim_i$.
\end{lem}
\begin{proof}
Suppose that we have a commutative square as below.
Then there exists a map $u_x : U \to X$ and a sequence of homotopies $h^1, \ldots h^n : U \to P(Y)$ between $f \circ u_x$ and $u$.
\[ \xymatrix{   \ar@{}[dr]|(.7){\sim^{r*}}            & X \ar[d]^f \\
              U \ar@{-->}[ur]^{u_x} \ar[d]_i \ar[r]_u & Y \ar[d]^g \\
              V \ar[r]_v                              & Z
            } \]
Then we have a sequence of homotopies $h^1_u, \ldots h^n_u$ between $g \circ f \circ u_x$ and $v \circ i$.
If $p_1$ has RLP with respect to $i$, then let $h^j_u = P(g) \circ h^j$; otherwise let $h^j_u = s \circ s \circ P(g) \circ h^j$.
By \rlem{hom-ext-rtc}, there exists a map $v_z : V \to Z$ and a sequence of homotopies $h^1_3, \ldots h^n_3 : V \to P(Z)$ between $v_z$ and $v$ such that $h^j_3 \circ i = h^j_u$.
Indeed, if $p_1$ has RLP with respect to $i$, then we can apply \rlem{hom-ext-rtc} to $p_1,p_0 : P(Z) \to Z$.
If we defined $h^j_u$ as $s \circ s \circ P(g) \circ h^j$, then we can apply \rlem{hom-ext-rtc} to $s \circ P(g) \circ h^j$ to get a sequence of homotopies $h^1_4, \ldots h^n_4$ between $v$ and $v_z$.
Then we can define $h^j_3$ as $s \circ h^j_4$.

By assumption on $g \circ f$, there exists a map $v_x : V \to X$ and a homotopy $h_2 : C_U(V) \to Z$ between $g \circ f \circ v_x$ and $v_z$.
Note that $h^1, \ldots h^n$ is a sequence of homotopies between $f \circ u_x = f \circ v_x \circ i$ and $u$.
Thus by \rlem{hom-ext-rtc}, we have a map $v_y : V \to Y$ and a sequence of homotopies $h^1_y, \ldots h^n_y : V \to P(Y)$ between $f \circ v_x$ and $v_y$ such that $h^j_y \circ i = h^j$.
In particular, $v_y \circ i = u$.
Thus we only need to prove that $g \circ v_y$ and $v$ are homotopic relative to $i$.
If $h^j_u = P(g) \circ h^j$, then let $h^j_1 = P(g) \circ h^j_y$.
If $h^j_u = s \circ s \circ P(g) \circ h^j$, then let $h^j_1 = s \circ s \circ P(g) \circ h^j_y$.
Then $h^j_1 \circ i = h^j_u$.
Thus we have a sequence of maps $[h^j_1,h^j_3] : V \amalg_U V \to P(Z)$.

If $p_0 \circ h^j_1 \sim_i p_0 \circ h^j_3$, then $p_1 \circ h^j_1 \sim_i p_1 \circ h^j_3$.
Indeed, if $h_0$ is a relative homotopy between $p_0 \circ h^j_1$ and $p_0 \circ h^j_3$, then consider the following diagram:
\[ \xymatrix{ V \amalg_U V \ar[rr]^{[h^j_1, h^j_3]} \ar[d]_{[\cyli_0,\cyli_1]} & & P(Z) \ar[d]^{p_0} \\
              C_U(V) \ar[rr]_{h_0} & & Z
            } \]
Since $p_0$ is pure with respect to $[\cyli_0,\cyli_1]$, we have a lift $q : C_U(V) \to P(Z)$.
Then $p_1 \circ q$ is a relative homotopy between $p_1 \circ h^j_1$ $p_1 \circ h^j_3$.

Finally, note that $h_2$ is a relative homotopy between $g \circ f \circ v_x = p_0 \circ h^1_1$ and $v_z = p_0 \circ h^1_3$.
It follows that $p_1 \circ h^n_1 = g \circ v_y$ and $p_1 \circ h^n_3 = v$ are also relatively homotopic.
\end{proof}

Now we can give a characterization of weak equivalences between fibrant objects.

\begin{prop}[min-we]
Let $\C$ be a model category, and let $\I$ be a class of maps of $\C$ such that $\Icof = \cof$.
Let $f : X \to Y$ be a morphism of $\C$ such that $X$ and $Y$ are fibrant.
Then the following conditions are equivalent:
\begin{enumerate}
\item $f$ is a weak equivalence.
\item $f$ has RLP up to $\sim_i$ with respect to every cofibration $i$.
\item $f$ has RLP up to $\sim_i$ with respect to every $i \in \I$.
\end{enumerate}
\end{prop}
\begin{proof}
$(1 \Rightarrow 2)$
Factor $f$ into a trivial cofibration $i : X \to Z$ followed by a trivial fibration $p : Z \to Y$.
Let $c : U \to V$ be a cofibration and let $u : U \to X$, $v : V \to Y$ be maps such that $f \circ u = v \circ c$.
Since $c$ is a cofibration and $p$ is a trivial fibration, we have a lift $q : V \to Z$.
Since $i$ is a trivial cofibration between fibrant objects, by \rprop{min-triv-cof}, it has a retraction $r : Z \to X$ such that $i \circ r \sim_i id_Z$.
Then $r \circ q$ is a required lift.

$(2 \Rightarrow 3)$ Obvious.

$(3 \Rightarrow 1)$
Factor $f$ into a trivial cofibration $f' : X \to Z$ followed by a fibration $g : Z \to Y$.
Let us prove that $f'$ has RLP up to $\sim^r$ with respect to every object.
By \rprop{min-triv-cof}, there exists a map $g' : Z \to X$ and a homotopy $f' \circ g' \sim_{f'} id_Z$.
By \rprop{path-cyl}, we have a right homotopy $h : Z \to P(Z)$ between $f' \circ g'$ and $id_Z$.
For every morphism $u : U \to Z$ we can define a lift $u' : U \to Y$ as $g' \circ u$.
Then $h \circ u$ is a homotopy between $f' \circ u'$ and $u$.

Since $Z$ and $Y$ are fibrant, conditions of \rlem{we-bot} are satisfied.
Hence $g$ has RLP up to $\sim_i$ with respect to every $i \in \I$.
Since $\J_\I$ consists of trivial cofibrations, by \rprop{triv-fib-iinj}, $g$ is a trivial fibration.
Thus $f$ is a weak equivalence by 2-out-of-3 property.
\end{proof}

Often the class of cofibrations is generated by much smaller class $\I$.
Thus the last condition in the previous proposition gives us a useful characterization of weak equivalences between fibrant objects which is easy to verify in practice.
In particular, if every object in a model category is fibrant, then this proposition gives us a complete characterization of weak equivalences, which we will use in the next section.

\subsection{Trivial cofibrations and fibrations with fibrant codomains}

If not every object of a model category is fibrant, then sometimes we can characterize fibrant object as those which have
RLP with respect to some set of trivial cofibrations $S$ which is considerably smaller than a set of generating cofibrations.
For example, Joyal model structure on simplicial sets has weak Kan complexes as fibrant objects,
which are simplicial sets that have RLP with respect to inner horns.
But inner horns do not generate the whole class of trivial cofibrations.

A natural candidate for a set of generating trivial cofibrations is $\J_\I \cup S$.
In general, this set is too small, but it works for trivial cofibrations and fibrations with fibrant codomains.
For simplicial sets (with Joyal model structure), a stronger version of this result was proved in \cite{lurie-topos} (where it is attributed to Joyal).

\begin{prop}
Let $\C$ be a model category.
Let $\I$ be a set of generating cofibrations, and let $S$ be a set of trivial cofibrations
such that every object that has RLP with respect to $S$ is fibrant.
Suppose that the domains and the codomains of maps in $\I$ and the domains of maps in $S$ are small relative to $\Icell$.

Then a map with fibrant codomain is a trivial cofibration if and only if it belongs to $\Icof[(\J_\I \cup S)]$.
A map with fibrant codomain is a fibration if and only if it belongs to $\Iinj[(\J_\I \cup S)]$.
\end{prop}
\begin{proof}
Let $f : X \to Z$ be a trivial cofibration such that $Z$ is fibrant.
Factor $f$ into a map $g : X \to Y \in \Icell[(\J_\I \cup S)]$ followed by a map $h : Y \to Z \in \Iinj[(\J_\I \cup S)]$.
\[ \xymatrix{ X \ar[r]^g \ar[d]_f & Y \ar[d]^h \\
              Z \ar@{=}[r] \ar@{-->}[ur] & Z
            } \]
By 2-out-of-3 property, $h$ is a weak equivalence.
Since it is a weak equivalence between fibrant objects, by \rprop{min-we}, it has RLP up to $i$ with respect to every $i \in \I$.
By \rprop{triv-fib-iinj}, $f \in \Iinj$.
Hence, we have a lift $Z \to Y$ as shown in the diagram above.
Thus $f$ is a retract of $g$, and it belongs to $\Icof[(\J_\I \cup S)]$.

Let $f : X \to Y$ be a map in $\Iinj[(\J_\I \cup S)]$ such that $Y$ is fibrant.
To prove that it is a fibration, we need to show that it has RLP with respect to every trivial cofibration.
Let $i : U \to V$ be a trivial cofibration, and let $u : U \to X$ and $v : V \to Y$ be maps such that the obvious square commutes.
Factor $v$ into a map $V \to Z \in \Icell[S]$ followed by a map $Z \to Y \in \Iinj[S]$.
\[ \xymatrix{ U \ar[rr]^u \ar[d]_i & & X \ar[d]^f \\
              V \ar[r] & Z \ar[r] \ar@{-->}[ur] & Y
            } \]
Since $Z$ is fibrant and $U \to Z$ is a trivial cofibration, we have a lift $Z \to X$ in the diagram above.
Thus $f$ has RLP with respect to $i$.
\end{proof}

We can try to use the previous proposition to characterize all trivial cofibrations and fibrations.
For example, if every map in $\Iinj[(\J_\I \cup S)]$ is a pullback of a map in $\Iinj[(\J_\I \cup S)]$ with fibrant codomain,
then $\J_\I \cup S$ generates the whole class of trivial cofibrations.
But we do not know any general situation when this property holds.

\section{Existence of model structures}

In this section, we will give necessary and sufficient conditions for the existence of a model structure in which all objects are fibrant.
Throughout this section let $\C$ be a fixed complete and cocomplete category and $\I$ a set of maps of $\C$
such that the domains and the codomains of maps in $\I$ are small relative to $\Icell$.
For every $i : U \to V \in \I$, choose a reflexive relative cylinder object $C_U(V)$
such that $[\cyli_0,\cyli_1] : V \amalg_U V \to C_U(V) \in \Icof$.
Let $\J_\I = \{\ \cyli_0 : V \to C_U(V)\ |\ i : U \to V \in \I \ \}$, and
let $\we_\I$ be the set of maps which have RLP up to $\sim_i$ with respect to every $i \in \I$.

We will consider the following conditions:
\begin{align}
& \text{For every composable $f \in \Jcell[_\I]$ and $g$, if $g \circ f \in \we_\I$, then $g \in \we_\I$} \label{cond:main} \tag{*} \\
& \text{For every composable $f \in \Jcell[_\I] \cup \we_\I$ and $g$, if $g \circ f \in \we_\I$, then $g \in \we_\I$} \label{cond:strong-main} \tag{*'}
\end{align}

\begin{lem}[main]
If condition~\eqref{cond:main} holds, then the following is true:
\begin{enumerate}
\item Every weak equivalence factors into a map in $\Jcell[_\I]$ followed by a map in $\Iinj$.
\item Every weak equivalence has RLP up to $\sim_i$ with respect to every $i \in \Icof$.
\item \label{it:we-top} For every $f : X \to Y$ and $g : Y \to Z$, if $g \in \we_\I$ and $g \circ f \in \we_\I$, then $f \in \we_\I$.
\item For every $f : X \to Y$ and $g : Y \to Z$, if $f \in \we_\I$ and $g \in \we_\I$, then $g \circ f \in \we_\I$.
\item $\Jcell[_\I] \subseteq \we_\I$.
\end{enumerate}
\end{lem}
\begin{proof}
Let $f : X \to Z$ be a weak equivalence.
Factor $f$ into maps $f' : X \to Y \in \Jcell[_\I]$ and $g : Y \to Z \in \Jinj[_\I]$.
By assumption, $g \in \we_\I$.
By \rprop{triv-fib-iinj}, $g \in \Iinj$.

Maps in $\Iinj$ are pure with respect to cofibrations.
Maps in $\Jcell[_\I]$ are split monomorphisms; thus they are pure with respect to all maps.
Hence every weak equivalence is pure with respect to cofibrations.
Since every weak equivalence factors into a map in $\Jcell[_\I]$ followed by a map in $\Iinj$,
to prove that every weak equivalence has RLP up to $\sim_i$ with respect to every cofibration $i : U \to V$,
it is enough to show that every map in $\Jcell[_\I]$ has this property.
Let $f : X \to Y$ be a map in $\Jcell[_\I]$.
It has a retraction $g : Y \to X$ which is a weak equivalence by condition~\eqref{cond:main}.
Hence $g$ is pure.
Let $u : U \to X$ and $v : V \to Y$ be maps such that the obvious square commutes.
Consider the following diagram:
\[ \xymatrix{ V \amalg_U V \ar[rr]^-{[f \circ g \circ v, v]} \ar[d]_{[\cyli_0,\cyli_1]} & & Y \ar[d]^g \\
              C_U(V) \ar[rr]_-{g \circ v \circ s} & & X
            } \]
Since $g$ is pure, we have a relative homotopy between $f \circ g \circ v$ and $v$.
Thus $g \circ v$ is a required lift in the original square.

Now, let us prove \eqref{it:we-top}.
Let $f : X \to Y$ and $g : Y \to Z$ be maps such that $g \in \we_\I$ and $g \circ f \in \we_\I$.
Consider the following diagram:
\[ \xymatrix{ U \ar[r]^u \ar[d]_i & X \ar[d]^f \\
              V \ar[r]^v \ar[rd]  & Y \ar[d]^g \\
                                  & Z
            } \]
Since $g \circ f \in \we_\I$, we have a lift $q : V \to X$ and a homotopy $h : C_U(V) \to Z$ between $g \circ f \circ q$ and $g \circ v$.
Consider the following diagram:
\[ \xymatrix{ V \amalg_U V \ar[r]^-{[f \circ q, v]} \ar[d]_{[\cyli_0,\cyli_1]} & Y \ar[d]^g \\
              C_U(V) \ar[r]_-h & Z
            } \]
Since $g$ is pure with respect to cofibrations, we have a lift $h' : C_U(V) \to Y$ which gives us a homotopy between $f \circ q$ and $v$.

Now, let us show that every $f : X \to Y \in \Jcell[_\I]$ is a weak equivalence.
Let $g : Y \to X$ be a retraction of $f$.
By assumption, $g$ is a weak equivalence.
By \eqref{it:we-top}, $f$ is a weak equivalence too.

Finally, let us prove that weak equivalences are closed under compositions.
To do this, it is enough to show that relation $\sim_i$ is transitive.
Let $h_0 : C_U(V) \to X$ be a homotopy between $f : V \to X$ and $f' : V \to X$, and
let $h_1 : C_U(V) \to X$ be a homotopy between $f' : V \to X$ and $f'' : V \to X$.
Then consider the following diagram:
\[ \xymatrix{ V \ar[r]^{\cyli_1} \ar[d]_{\cyli_0} & C_U(V) \ar[d]^{p_0} \\
              C_U(V) \ar[r]_{p_1} \ar[d]_s & \po Z \ar[d]^q \\
              V \ar[r]_{\cyli_1} & \po C_U(V)
            } \]
Since $p_0 \in \Jcell[_\I]$, $q$ is a weak equivalence.
Consider the following diagram:
\[ \xymatrix{ V \amalg_U V \ar[rr]^-{[p_0 \circ \cyli_0, p_1 \circ \cyli_1]} \ar[d]_{[\cyli_0,\cyli_1]} & & Z \ar[d]^q \\
              C_U(V) \ar@{=}[rr] & & C_U(V)
            } \]
Since $q$ is pure with respect to cofibrations, we have a lift $h : C_U(V) \to Z$.
Then $[h_0,h_1] \circ h$ is a homotopy between $f$ and $f''$.
\end{proof}

\begin{thm}[main]
There exists a model structure on $\C$ with $\Icof$ as the class of cofibrations and $\Jcof[_\I]$ as the class of trivial cofibrations
if and only if condition~\eqref{cond:strong-main} holds.
\end{thm}
\begin{proof}
First, suppose that such model structure on $\C$ exists.
Then every object is fibrant and $C_U(V)$ is a correct cylinder object.
By \rprop{min-we}, $\we_\I$ is the class of weak equivalences.
Hence condition~\eqref{cond:strong-main} holds.

Now, suppose that condition~\eqref{cond:strong-main} holds.
Let us verify the conditions of \rprop{model-cat}:
\begin{description}
\item[(A1)] This holds by assumption.
\item[(A2)] The closure of $\we_\I$ under retracts is obvious.
One part of 2-out-of-3 property holds by assumption, other two parts follow from \rlem{main}.
\item[(A3)] This is true since $C_U(V)$ is reflexive.
\item[(A4)] This holds by \rlem{main}.
\item[(A5)] This holds by \rprop{triv-fib-iinj}.
\end{description}
\end{proof}

If maps in $\I$ satisfy some mild assumptions, then we can simplify the condition in \rthm{main}:
\begin{prop}[main]
Suppose that the domains of maps in $\I$ are cofibrant.
Then conditions \eqref{cond:main} and \eqref{cond:strong-main} are equivalent.
\end{prop}
\begin{proof}
Condition~\eqref{cond:strong-main} obviously implies condition~\eqref{cond:main}.
Let us prove the converse.
Let $f : X \to Y$ and $g : Y \to Z$ be maps such that $f \in \we_\I$ and $g \circ f \in \we_\I$.
By \rlem{main}, we can factor $f$ into maps $f' : X \to X' \in \Jcell[_\I]$ and $g' : X' \to Y \in \Iinj$.
Then $g \circ g' \in \we_\I$ by assumption.
Consider the following diagram:
\[ \xymatrix{ & X' \ar[d]^{g'} \\
              U \ar[r]_u \ar[d]_i \ar@{-->}[ur]^{u'} & Y \ar[d]^g \\
              V \ar[r]_v & Z
            } \]
Since $U$ is cofibrant, we have a lift $u' : U \to X'$.
Since $g \circ g' \in \we_\I$, we have a lift $v' : V \to X'$ such that $g \circ g' \circ v' \sim_i v$.
Then $g' \circ v'$ is a required lift in the original square.
\end{proof}

Thus the main problem is to verify condition~\eqref{cond:main}.
There are a few ways to do this, but the idea is the same:
we need to assume that there exists some notion of homotopy on sets of maps which satisfies some conditions.
There are two standard ways to do this: using path and cylinder objects.
Now, we present this constructions.

\begin{prop}[main-path]
Condition~\eqref{cond:main} holds if and only if for every object $X$,
there exists a path object $P(X)$ such that the following conditions hold:
\begin{enumerate}
\item For every $f : X \to Y$, there exists a morphism of path objects $(f,P(f)) : P(X) \to P(Y)$,
\item Either $p_1$ has RLP with respect to $\I$ or there exists a map $s : P(X) \to P(X)$ such that $p_0 \circ s = p_1$ and $p_1 \circ s = p_0$.
\item $p_0$ has RLP with respect to $\I$.
\item \label{it:either} Either path objects are reflexive and maps $\langle p_0, p_1 \rangle : P(X) \to X \times X$ have RLP with respect to $\J_\I$
or maps in $\Jcell[_\I]$ have RLP up to $\sim^r$ with respect to the domains of maps in $\I$.
\end{enumerate}
\end{prop}
\begin{proof}
If condition~\eqref{cond:main} holds, then we can define path objects as usual using a factorization
of the diagonal $X \to X \times X$ into maps $X \to P(X) \in \Jcell[_\I]$ and $P(X) \to X \times X \in \Jinj[_\I]$.

Note that if the second condition of \eqref{it:either} holds, then condition~\eqref{cond:main} holds by \rlem{we-bot}.
Thus we only need to prove that the first condition of \eqref{it:either} implies the second.
Indeed, let $f : X \to Y$ be a map in $\Jcell[_\I]$, and let $g : Y \to X$ be its retraction.
Consider the following diagram:
\[ \xymatrix{ X \ar[rr]^-{t \circ f} \ar[d]_f & & P(Y) \ar[d]^{\langle p_0, p_1 \rangle} \\
              Y \ar[rr]_-{\langle f \circ g, id_Y \rangle} \ar@{-->}[urr] & & Y \times Y
            } \]
We have a lift $h : Y \to P(Y)$ which gives us a right homotopy between $f \circ g$ and $id_Y$.
Now, for every $u : U \to Y$, we can define a lift $u' = g \circ u : U \to X$ and a homotopy $h \circ u$ between $f \circ u'$ and $u$.
Thus $f$ has RLP up to $\sim^r$ with respect to any object.
\end{proof}

\begin{exmp}
An example of a model category defined in this way is a folk model structure on the category of $\omega$-categories which was constructed in \cite{folk}.
The conditions of \rprop{main-path} follow from the results of \cite{folk}, but some of them are not needed for this proposition.
Thus the construction of this model structure can be somewhat simplified using the general results of this section.
\end{exmp}

Instead of path objects, we could try to use cylinder objects to verify condition~\eqref{cond:main}.
The advantage of this approach is that we do not need to define a cylinder object for every object of the category,
only for objects that are domains and codomains of generating cofibrations.
The disadvantage is that we still need to verify that maps in $\Jcell$ has RLP up to $\sim$ with respect to domains of generating cofibrations.

\begin{prop}[main-cyl]
Condition~\eqref{cond:main} holds if and only if for every object $X$ which either domain or codomain of a map in $\I$,
there exists a cylinder object $C(X)$ such that the following conditions hold:
\begin{enumerate}
\item For every $i : U \to V \in \I$, there exists a morphism of cylinder objects $(i,C(i)) : C(U) \to C(V)$.
\item There exists a map $s : C(X) \to C(X)$ such that $s \circ i_0 = i_1$, $s \circ i_1 = i_0$, and $C(i) \circ s = s \circ C(i)$.
\item These cylinder objects satisfy homotopy extension property. That is,
if $i : U \to V \in \cof$, $u : U \to X$ and $v : V \to X$ are maps, and $h : C(U) \to X$ is a homotopy between $v \circ i$ and $u$,
then there exists a map $v' : V \to X$ and a homotopy $h' : C(V) \to X$ between $v$ and $v'$ such that $h = h' \circ C(i)$.
\item Maps in $\Jcell[_\I]$ have RLP up to $\sim^*$ (reflexive transitive closure of $\sim$) with respect to domains of maps in $\I$.
\item For every $i : U \to V \in \I$, we have a lift $p$ in the following diagram:
\[ \xymatrix{ V \amalg_U V \ar[r]^-f \ar[d]_{[\cyli_0,\cyli_1]} & T \\
              C_U(V) \ar@{-->}[ur]_p
            } \]
where $T = C_U(V) \amalg_{(V \amalg_U V)} (C(V) \amalg_{C(U)} C(V))$ is the pushout of maps $[\cyli_0,\cyli_1] : V \amalg_U V \to C_U(V)$
and $\cyli_0 \amalg_{\cyli_0} \cyli_0 : V \amalg_U V \to C(V) \amalg_{C(U)} C(V)$,
and $f : V \amalg_U V \to T$ is the composite $V \amalg_U V \xrightarrow{\cyli_1 \amalg_{\cyli_1} \cyli_1} C(V) \amalg_{C(U)} C(V) \to T$.
\end{enumerate}
\end{prop}
\begin{proof}
If condition~\eqref{cond:main} holds, then we can define such cylinder objects as usual by factorization of $[id_X,id_X] : X \amalg X \to X$ into a cofibration and a trivial fibration.
The proof of the converse is similar to the proof of \rlem{we-bot}.

Suppose that we have a commutative square as below, where $f \in \Jcell[_\I]$ and $g \circ f \in \we_\I$.
Then there exists a map $u_x : U \to X$ and a sequence of homotopies $h^1, \ldots h^n : C(U) \to Y$ between $f \circ u_x$ and $u$.
\[ \xymatrix{   \ar@{}[dr]|(.7){\sim^*}               & X \ar[d]^f \\
              U \ar@{-->}[ur]^{u_x} \ar[d]_i \ar[r]_u & Y \ar[d]^g \\
              V \ar[r]_v                              & Z
            } \]
Then we have a sequence of homotopies $g \circ h^1 \circ s, \ldots g \circ h^n \circ s$ between $v \circ i$ and $g \circ f \circ u_x$.
By homotopy extension property, there exists a map $v_z : V \to Z$ and a sequence of homotopies $h^1_3, \ldots h^n_3 : C(V) \to Z$ between $v$ and $v_z$ such that $h^j_3 \circ C(i) = g \circ h^j \circ s$.

Since $g \circ f \in \we_\I$, there exists a map $v_x : V \to X$ and a homotopy $h_2 : C_U(V) \to Z$ between $g \circ f \circ v_x$ and $v_z$.
Note that $h^1 \circ s \circ s, \ldots h^n \circ s \circ s$ is a sequence of homotopies between $f \circ u_x = f \circ v_x \circ i$ and $u$.
Thus by homotopy extension property, we have a map $v_y : V \to Y$ and a sequence of homotopies $h^1_y, \ldots h^n_y : C(V) \to Y$ between $f \circ v_x$ and $v_y$ such that $h^j_y \circ C(i) = h^j \circ s \circ s$.
In particular, $v_y \circ i = u$.
Thus we only need to prove that $g \circ v_y$ and $v$ are homotopic relative to $i$.
If we let $h^j_1 = g \circ h^j_y$, then $h^j_1 \circ C(i) = g \circ h^j \circ s \circ s = h^j_3 \circ s \circ C(i)$.
Thus we have a sequence of maps $[h^j_1, h^j_3 \circ s] : C(V) \amalg_{C(U)} C(V) \to Z$.

If $[h^j_1, h^j_3 \circ s] \circ (\cyli_0 \amalg_{\cyli_0} \cyli_0) : V \amalg_U V \to Z$ extends to $C_U(V)$, then $[h^j_1, h^j_3 \circ s] \circ (\cyli_1 \amalg_{\cyli_1} \cyli_1)$ also extends to $C_U(V)$.
Indeed, if $[h^j_1, h^j_3 \circ s] \circ (\cyli_0 \amalg_{\cyli_0} \cyli_0) = h_0 \circ [\cyli_0,\cyli_1]$ for some $h_0 : C_U(V) \to Z$, then by assumption we have a map $q : T \to Z$
such that $q \circ p : C_U(V) \to Z$ is an extension of $[h^j_1, h^j_3 \circ s] \circ (\cyli_1 \amalg_{\cyli_1} \cyli_1)$.

Finally, note that $h_2$ is an extension of $[g \circ f \circ v_x, v_z] = [h^1_1, h^1_3 \circ s] \circ (\cyli_0 \amalg_{\cyli_0} \cyli_0)$.
It follows that we have an extension of $[h^n_1, h^n_3 \circ s] \circ (\cyli_1 \amalg_{\cyli_1} \cyli_1) = [g \circ v_y, v]$,
which defines a relative homotopy between $g \circ v_y$ and $v$.
\end{proof}

\begin{exmp}
An example of a model category defined in this way is the usual model structure on topological spaces.
If we define $C(X)$ as $I \times X$, then the conditions of \rprop{main-cyl} are easy to verify directly.
\end{exmp}

\bibliographystyle{amsplain}
\bibliography{ref}

\end{document}